\documentclass[a4paper,11pt]{article}

\usepackage[margin=2.5cm]{geometry}

\usepackage{amsmath,amssymb,amsthm}
\usepackage{xfrac}

\usepackage[numbers,sort&compress]{natbib}
\usepackage{enumerate}
\usepackage[colorlinks=true,linkcolor=blue,citecolor=blue,urlcolor=blue]{hyperref}

\newtheorem{mainthm}{Theorem}
\newtheorem{theorem}{Theorem}[section]
\newtheorem{lemma}[theorem]{Lemma}
\newtheorem{proposition}[theorem]{Proposition}
\newtheorem{corollary}[theorem]{Corollary}

\theoremstyle{definition}
\newtheorem{definition}[theorem]{Definition}
\newtheorem{example}[theorem]{Example}

\theoremstyle{remark}
\newtheorem{remark}[theorem]{Remark}

\newcommand{\F}{\mathbb{F}}
\newcommand{\ov}[1]{\overline{#1}}
\DeclareMathOperator{\Tr}{Tr}
\DeclareMathOperator{\End}{End}

\title{\textbf{Twisted group codes of dimension 2 and 3}}

\author{%
  \'Alvaro Otero S\'anchez%
  \thanks{The author has been supported by the Spanish Ministry of Science, Innovation and Universities under the FPU 2023 grant program.} \\
  \small Department of Mathematics, Universidad de Almer\'ia, Almer\'ia, Spain \\
  \small \texttt{aos073@ual.es}%
}

\date{}

\begin{document}

\maketitle

\begin{abstract}
It is known that all group codes of dimension at most $3$ are permutation equivalent to abelian group codes. We prove that twisted group codes of dimension at most $3$ are monomially equivalent either to an abelian group code or to an irreducible constacyclic code, and in the second case we describe the code explicitly by means of a trace map. Monomial equivalence cannot be replaced by permutation equivalence, and the second case cannot be avoided, as we present a twisted group code that is not monomially equivalent to any group code. More generally, we give two infinite families, in dimensions $2$ and $3$, of irreducible constacyclic codes which satisfy the Plotkin bound, whereas every group code with the same length and dimension has smaller minimum distance.
\end{abstract}

\vspace{1ex}
\noindent\textbf{Keywords:} Group code, Twisted group algebra, Abelian group code \\
\noindent\textbf{MSC (2020):} 94B05, 94B60, 16S35

\setcounter{section}{-1}
\section{Introduction}
One of the modern tools to produce new linear codes with desirable properties are the so called group codes. Let $G$ be a group and $F$ be a field. Then the group ring $FG$ is the formal linear combination of elements of $G$, where the product is given by the group operation in $G$. This ring was initially studied as it has a strong connection with representation theory, where linear representations of a group are in bijection with modules over the group algebras. A group code, or $G-$code, is an ideal over a group algebra, which is usually defined with $G$ a finite group and $F$ a finite field of prime characteristic. 

In \cite{Berman} Berman proved that binary Reed Muller codes are ideals in a binary group algebra over an abelian group. In the following decade some mathematicians studied ideals over group algebras as they provide a rich framework to obtain new linear codes with interesting properties. 

Bernal, del R\'io and Sim\'on showed that one-dimensional group codes are abelian \cite{BdRS}, in the sense that they are permutation equivalent to a $G$-code with $G$ abelian. This result was extended by Garc\'ia Pillado, Gonz\'alez, Markov, Markova and Mart\'inez, who proved that the same is true in dimensions $2$ and $3$ \cite{GGMMM}. Since there are non-abelian group codes of dimension $4$ over every field of odd characteristic \cite{GGMM}, this settles the question of the smallest dimension of a non-abelian group code.

In this paper we consider the same question for twisted group codes. Let $F$ be a finite field and $G$ a finite group. Recall that a map $\alpha : G\times G\to F^*$ is a (normalized) $2$-cocycle if
\begin{equation}\label{eq:cocycle}
\alpha(x,y)\,\alpha(xy,z)=\alpha(y,z)\,\alpha(x,yz)\quad\text{and}\quad\alpha(x,1)=\alpha(1,x)=1
\end{equation}
for all $x,y,z\in G$. The twisted group algebra $F^\alpha G$ is the vector space with basis $\{\ov g:g\in G\}$ and multiplication $\ov g\,\ov h=\alpha(g,h)\,\ov{gh}$, extended bilinearly. Note that condition \eqref{eq:cocycle} says precisely that this multiplication is associative with identity $\ov1$. 

The ideals of $F^\alpha G$, regarded as codes in $F^G$ through the coordinates with respect to the basis $\{\ov g\}$, are the twisted group codes of de la Cruz and Willems \cite{dlCW}. For a cyclic group they are exactly the constacyclic codes (Example \ref{ex:consta}), so twisted group codes are a common generalization of group codes and constacyclic codes.

Monomially equivalent twisted group codes arise naturally in the study of twisted group algebras. A rescaling of the basis $\ov g\mapsto t(g)\,\ov g$ transforms $F^\alpha G$ into another twisted group algebra and every ideal into a monomially equivalent one (Proposition \ref{prop:graded}), and already in dimension one there are twisted group codes which are monomially, but not permutation, equivalent to group codes (Remark \ref{rem:perm}). So the natural question is whether every twisted group code of dimension at most $3$ is monomially equivalent to an abelian group code. To do so, we present a study of the twisted group codes of small dimension.

\begin{mainthm}\label{thm:main}
Let $I\not=0$ be an ideal of $F^\alpha G$ with $\dim_FI\le3$.
\begin{enumerate}[(a)]
\item If $\alpha$ is a coboundary, then $I$ is monomially equivalent to an abelian group code.
\item If $\alpha$ is not a coboundary, then $m=\dim_FI\in\{2,3\}$, and there are a field $L\supseteq F$ with $[L:F]=m$ and an $\alpha$-character $\rho\colon G\to L^*$ such that
\[
I=\Big\{\sum_{g\in G}\Tr_{L/F}\big(a\,\rho(g)^{-1}\big)\,\ov g\ :\ a\in L\Big\}.
\]
Moreover, $I$ is monomially equivalent to an irreducible constacyclic code of length $|G|$, that is, to a minimal ideal of $F[x]/(x^{|G|}-c)$ for some $c\in F^*$.
\end{enumerate}
\end{mainthm}

Therefore, every twisted group code of dimension at most $3$ is monomially equivalent to an abelian group code or to an irreducible constacyclic code. Remark \ref{rem:perm} shows that monomial equivalence cannot be replaced by permutation equivalence. A natural question is whether the second case can be avoided, that is, all twisted group codes are monomially equivalent to a group code. We prove that it is not the case, as the following result shows.

\begin{mainthm}\label{thm:tetra}
The ternary tetracode, of parameters $[4,2,3]_3$, is an ideal of the commutative twisted group algebra $\F_3[x]/(x^4+1)$ of the cyclic group of order $4$, but it is not monomially equivalent to any group code.
\end{mainthm}

This code is not an isolated example, as we are able to provide a family of twisted group codes with better parameters than any group code.

\begin{mainthm}\label{thm:families}
Let $k\in\{2,3\}$, let $r\ge1$ and let $n=rN_k$ with  $N_k=(q^k-1)/(q-1)$.
\begin{enumerate}[(i)]
\item There is a minimal ideal of $F[x]/(x^n-c)$, for some $c\in F^*$, which is a code of length $n$, dimension $k$ and minimum distance $rq^{k-1}$.
\item If $k$ does not divide $q-1$, there is also a minimal ideal of $FC_n$ of length $n$, dimension $k$ and minimum distance $rq^{k-1}$.
\item If $k$ divides $q-1$ and $k$ does not divide $r$, then every ideal of a group algebra $FG$ with $|G|=n$ and dimension $k$ has minimum distance at most $rq^{k-1}-1$. Consequently, the code of (i) is not monomially equivalent to any group code.
\end{enumerate}
\end{mainthm}

For $q=3$, $k=2$ and $r=1$ the code of Theorem \ref{thm:families} is the tetracode, so Theorem \ref{thm:tetra} is a particular case of Theorem \ref{thm:families}.

The codes appearing in Theorem \ref{thm:main}(b), which we call twisted trace codes, make sense in any dimension, and we show that they are always monomially equivalent to irreducible constacyclic codes (Proposition \ref{prop:cyclic}). As in \cite{GGMMM}, our results concern two-sided ideals. For one-sided ideals the situation is different already for group algebras, since there are non-abelian left group codes of dimension $2$ \cite[Proposition 3.3]{BdRS}. In Example \ref{ex:quaternion} we give a one-sided twisted group code of length $4$ and dimension $2$ whose minimum distance exceeds that of every group code with the same length and dimension.

The paper is organized as follows. In Section \ref{sec:prelim} we recall some basic definitions and prove some auxiliary results. In Section \ref{sec:trace} we introduce twisted trace codes and prove Theorem \ref{thm:main}. In Section \ref{sec:examples} we prove Theorems \ref{thm:families} and \ref{thm:tetra}, and we give some further examples. Finally, in Section \ref{sec:concl} we discuss some open questions.

\section{Preliminaries}\label{sec:prelim}

For the rest of the paper, let $F=\F_q$ be a finite field, let $G$ be a finite group and let $\alpha$ be a normalized $2$-cocycle of $G$ with values in $F^*$. We identify an element $y=\sum_gy_g\,\ov g$ of $F^\alpha G$ with the vector $(y_g)_{g\in G}\in F^G$, so that supports, weights and minimum distances are defined as usual.

Two codes $\mathcal C\subseteq F^X$ and $\mathcal C'\subseteq F^Y$, with $|X|=|Y|$, are monomially equivalent \cite{HuffmanPless2003} if there exist a bijection $\pi\colon Y\to X$ and scalars $\kappa_y\in F^*$ such that $\mathcal C'=\{(\kappa_y\,c_{\pi(y)})_{y\in Y}:c\in\mathcal C\}$, and permutation equivalent if this holds with all $\kappa_y=1$. Monomially and permutation equivalent codes have the same length, dimension and weight distribution.

From the identities of \eqref{eq:cocycle} we can prove that for all $g\in G$ and $y\in F^\alpha G$, with $y=\sum_{g\in G} y_g \overline{g}$ we have that
\begin{equation}\label{eq:coef}
\ov g\,y=\sum_{x\in G}\alpha(g,g^{-1}x)\,y_{g^{-1}x}\,\ov x,\qquad\alpha(g,g^{-1})=\alpha(g^{-1},g).
\end{equation}
In particular every $\ov g$ is a unit of $F^\alpha G$, and the coefficient of $\ov1$ in $\ov{g^{-1}}\,y$ is $\alpha(g^{-1},g)\,y_g$, which will be used in following sections.

\begin{definition}
    Let $G$ be a group, and $F$ be a field. A $2-$cocycle $\alpha\in Z^2(G,F^*)$ is said to be a $2-$coboundary if there exists a map $t:G\longrightarrow F^*$ such that
    \begin{equation*}
        \alpha(g,h) = t(g)t(h) t(gh)^{-1} \quad \forall g,h\in G
    \end{equation*}
    The set of all $2-$coboundaries will be denoted by $B^2(G,F^*)$. If $L$ is a finite extension of $F$, a map $\rho\colon G\to L^*$ is called an $\alpha$-character if $\rho(g)\rho(h)=\alpha(g,h)\rho(gh)$ for all $g,h\in G$. When the context is clear, we will simply refer to $2$-cocycles as cocycles.
\end{definition}

The following result proves that coboundaries play a crucial role in the structure of twisted group algebras. For further details on the cohomology of groups and its computation, see \cite{Brown1982}. 
\begin{proposition}[\cite{Linckelmann2018vol2}, Proposition 1.2.6]
     Let $G$ be a group, and $F$ be a field. Then
     \begin{enumerate}
     \renewcommand{\labelenumi}{(\roman{enumi})}
         \item If $\alpha,\beta\in Z^2(G,F^*)$, then the product given by 
         \begin{align*}
             \alpha\beta: G\times G &\longrightarrow F^* \\
             (g,h) &\longmapsto \alpha(g,h) \beta(g,h)
         \end{align*}
     
     gives to $Z^2(G,F^*)$ the structure of abelian group.
     \item $B^2(G,F^*)$ is a normal subgroup of $Z^2(G,F^*)$, and we will call it quotient $H^2(G,F^*)= Z^2(G,F^*)/B^2(G,F^*)$ the second cohomology group of $G$ with coefficient in $F$.      
     \item Let  $\alpha,\beta\in Z^2(G,F^*)$. Then $F^\alpha G \simeq F^\beta G$ as graded algebras if and only if $[\alpha] = [\beta]$ in $H^2(G,F^*)$.
     \item Let  $\alpha,\beta\in Z^2(G,F^*)$. Then $[\alpha] = [\beta]$ in $H^2(G,F^*)$. if and only if there exist a map $t:G\longrightarrow F^*$ such that $\alpha(g,h)=\beta(g,h)\,t(g)t(h)t(gh)^{-1}$ for all $g,h\in G$.
     \end{enumerate}
\end{proposition}

We say that $\alpha,\beta \in Z^2(G,F^*)$ are cohomologous if $[\alpha] = [\beta]$ in $H^2(G,F^*)$.  This play a crucial role in code theory, as the following result states that cohomologous cocycles gives monomially equivalent codes.

\begin{proposition}\label{prop:graded}
Let $\alpha,\beta$ be cocycles of $G$ and let $t\colon G\to F^*$ a map with $t(1)=1$. The linear map $\Phi_t\colon F^\alpha G\to F^\beta G$, $\ov g\mapsto t(g)\,\ov g$, is an isomorphism of algebras if and only if $\alpha(g,h)=\beta(g,h)\,t(g)t(h)t(gh)^{-1}$ for all $g,h\in G$. In this case $\Phi_t$ maps left (right, two-sided) ideals onto left (right, two-sided) ideals, and every code onto a monomially equivalent code. In particular, if $\alpha$ is a coboundary, then every ideal of $F^\alpha G$ is monomially equivalent to an ideal of $FG$ of the same kind.
\end{proposition}

\begin{proof}
We have the following identities.
\begin{align*}
    \Phi_t(\overline{g})\Phi_t(\overline{h}) &= t(g)t(h)\,\overline{g}\overline{h} = t(g)t(h)\beta(g,h)\,\overline{gh} \\
    \Phi_t(\overline{g}\,\overline{h}) &= \Phi_t(\alpha(g,h)\,\overline{gh}) = \alpha(g,h)t(gh)\,\overline{gh}
\end{align*}
Therefore both identities are equal if and only if $\alpha(g,h)=\beta(g,h)\,t(g)t(h)t(gh)^{-1}$, which is equivalent to say that $\Phi_t$ is an algebra isomorphism. The remaining assertions are clear, since $\Phi_t$ multiplies the coordinate indexed by $g$ by $t(g)$. \qed
\end{proof}

\begin{example}\label{ex:consta}
Let $G=\langle g\rangle$ be a cyclic group of order $n$ and let $c\in F^*$. For $0\le i,j<n$ we define $\alpha(g^i,g^j)=1$ if $i+j<n$ and $\alpha(g^i,g^j)=c$ otherwise. It is easy to check that it is a cocycle, and that the map $x^i\mapsto\ov{g^i}$ is an isomorphism of algebras $F[x]/(x^n-c)\to F^\alpha G$. Therefore, the ideals of $F^\alpha G$ are precisely the constacyclic codes of length $n$. In addition, it is well known \cite{Brown1982} that all cohomology classes of a cyclic group have a representative of the form of $\alpha$. Therefore, every twisted group code of a cyclic group is monomially equivalent to a constacyclic code.
\end{example}

The following easy observation give us the structure of one dimensional twisted group codes. This will be used to study the codes of dimension $2$ and $3$ over a twisted group algebra.

\begin{lemma}\label{lem:onedim}
$F^\alpha G$ has a left or right module $M$ with $\dim_FM=1$ if and only if $\alpha$ is a coboundary.
\end{lemma}

\begin{proof}   

    The if part is trivial as $\alpha$ is a coboundary implies that $F^\alpha G \simeq FG$ and we can take $M$ the trivial module. For the other direction, w.l.o.g. let assume that $M$ is a left module, and let $g\in G$. As we have that $\dim_FM=1$, then $M=Fv$ for some $v\in M$. In addition, $M$ is a left $F^\alpha G$-module, so $\overline{g}M=M$. Therefore we have that $\overline{g}v = \lambda(g)v$ for some $\lambda(g) \in F$. Note that $\lambda(g)\not=0$ as $\overline{g}$ is a unit, and $\lambda(1)=1$.

    We have that $\overline{g}\overline{h} = \alpha(g,h)\,\ov{gh}$, so, if we look at it action  over $v$ we obtain $\lambda(g)\lambda(h)=\alpha(g,h)\lambda(gh)$, that is, $\alpha(g,h)=\lambda(g)\lambda(h)\lambda(gh)^{-1}$ and therefore it is a coboundary. \qed
\end{proof}

Finally, we have characterised the ideals of dimension one over $F^\alpha G$, which are precisely the monomially equivalent to the repetition code. This is an analogue of \cite[Corollary 1.6]{GGMMM} over twisted group algebras.

\begin{proposition}\label{prop:dim1}
Let $I=Fv$ be a one-dimensional left ideal of $F^\alpha G$. Then $\alpha$ is a coboundary, all the coordinates of $v$ are nonzero, and $I$ is monomially equivalent to the repetition code of length $|G|$, which is an ideal of $FH$ for every group $H$ of order $|G|$.
\end{proposition}

\begin{proof}
    First of all, we know by Lemma \ref{lem:onedim} that $\alpha$ is a coboundary and $\ov g\,v=\lambda(g)v$ with $\lambda(g)\in F^*$. In addition, by \eqref{eq:coef}, the coefficient of $\ov g$ in $\ov g\,v$ is $v_1$, and hence $v=v_1\sum_g\lambda(g)^{-1}\ov g$ with $v_1\ne0$. 

    Let $\Phi_\lambda: F^\alpha G \longrightarrow FG$ given by $\overline{g} \longmapsto \lambda(g)\overline{g}$. By Proposition \ref{prop:graded}, $\Phi_\lambda$ is an isomorphism of algebras and $I$ is monomially equivalent to $\Phi_\lambda(v) =v_1\sum_g\lambda(g)\lambda(g)^{-1}\ov g = v_1 \sum_g \overline{g}$. As a result, $I$ is monomially equivalent to the repetition code of $|G|$. 
    
    Finally, the repetition code is present in $FH$ for all $|H|=|G|$ as it corresponds to the ideal of $FH$ spanned by $\sum_{h\in H}h$. \qed
\end{proof}

\begin{remark}\label{rem:perm}
Note that in Proposition \ref{prop:dim1} we can not replace monomial equivalence by permutation equivalence. For example, let $F=\F_5$, $G=\langle g\rangle\cong C_2$ and let $\alpha$ be the cocycle given by $\ov g^{\,2}=4\cdot\ov1$, that is, $\alpha(g,g)=4$.

First of all, we compute the one-dimensional ideals of $F^\alpha G$. By the proof of Proposition \ref{prop:dim1}, we have that $\lambda(g)^2=\alpha(g,g)\lambda(1)=4$, so $\lambda(g)=2$ or $\lambda(g)=3$. Since $2^{-1}=3$ in $\F_5$, the one-dimensional ideals of $F^\alpha G$ are spanned by $1+3\ov g$ and $1+2\ov g$. On the other hand, the only group of order $2$ is $C_2$, and in $\F_5C_2$ we have that $\lambda(g)^2=1$, so $\lambda(g)=\pm1$ and the one-dimensional ideals are spanned by $1+\ov g$ and $1+4\ov g$.

Now, the only nontrivial permutation of the coordinates interchanges the coefficients of $\ov 1$ and $\ov g$. Therefore it transforms $1+2\ov g$ into $2+\ov g=2(1+3\ov g)$, and hence it interchanges the codes $\langle1+2\ov g\rangle$ and $\langle1+3\ov g\rangle$. As a result, a permutation maps each of these two codes onto one of them, and neither is spanned by $1+\ov g$ or $1+4\ov g$. Finally, we conclude that neither $\langle1+2\ov g\rangle$ nor $\langle1+3\ov g\rangle$ is permutation equivalent to a group code.
\end{remark}

\section{Twisted trace codes}\label{sec:trace}
In this section we study the twisted group codes of dimension at most $3$. We prove that each of them is monomially equivalent either to an abelian group code or to an irreducible constacyclic code. In the second case we prove that we can enlarge the base field making the code exactly a twisted trace code.

For the rest of the section, let $L\supseteq F$ be a field with $[L:F]=m$ for some $m\in \mathbb{N}$ and let $\rho\colon G\to L^*$ be an $\alpha$-character. We will prove that in such field $L$ all twisted group codes of dimension $2$ and $3$ are precisely the twisted trace code.

We will denote by $F[\rho(G)]$ the $F$-span generated by the elements  of $\rho(G)$. It is easy to see that  $\rho(1)=1$, and for all $g,h\in G$, we have that $ \rho(g)\rho(h)=\alpha(g,h)\rho(gh)\in F^*\rho(gh).$ Therefore, $F[\rho(G)]$ is closed under multiplication and contains the identity $1=\rho(1)$, so it is a subring of $L$. Moreover, $L$ is a finite field, all its subring are fields, and in particular $F[\rho(G)]$ is field. 

W.l.o.g we can assume that $L$ is exactly $F[\rho(G)]$. Now we can introduce the twisted trace code.

\begin{definition}
    Let $L=F[\rho(G)]$, then let $t_\rho(a)=\sum_{g\in G}\Tr_{L/F}\big(a\,\rho(g)^{-1}\big)\,\ov g$. Then the twisted trace code, denoted by $T_\rho$, is defined as 
    \[
T_\rho=\{t_\rho(a):a\in L\} = \Big\{\sum_{g\in G}\Tr_{L/F}\big(a\,\rho(g)^{-1}\big)\,\ov g:a\in L\Big\}.
\]
\end{definition}

Note that the trace $\Tr_{L/F}$ takes values in
$F$, so every coefficient of $t_\rho(a)$ lies in $F$, and therefore
$t_\rho(a)\in F^\alpha G$. Moreover, the trace is a $F$-linear map, so $t_\rho\colon L\to F^\alpha G$ is also
$F$-linear and therefore $T_\rho$ is a linear code of length $|G|$ over $F$. Now we will prove that it is a twisted group code, that is, it is two-sided ideal of $F^\alpha G$.

\begin{lemma}\label{lem:tro}
Let $a\in L$ and let $h\in G$. Then $
\ov h\,t_\rho(a)=t_\rho(a)\,\ov h=t_\rho\big(a\rho(h)\big) $
\end{lemma}
\begin{proof}
Let $a\in L$ and $h,w\in G$. We know by equation \eqref{eq:coef} that the coefficient of $\ov w$ in
$\ov h\,t_\rho(a)$ is given by 
\[(\ov h\,t_\rho(a))_w = \alpha(h,h^{-1}w)\,(\,t_\rho(a))_{h^{-1}w} =\alpha(h,h^{-1}w)\Tr_{L/F}\big(a\rho(h^{-1}w)^{-1}\big)\]

Note that $\rho(h)\rho(h^{-1}w)=\alpha(h,h^{-1}w)\rho(w)$, and therefore we have that 
\begin{align*}
    (\ov h\,t_\rho(a))_w & = \alpha(h,h^{-1}w)\Tr_{L/F}\big(a\rho(h^{-1}w)^{-1}\big) \\
    & = \alpha(h,h^{-1}w)\Tr_{L/F}\big(a\alpha(h,h^{-1}w)^{-1}\rho(h)\rho(w)^{-1}\big) \\
    & =\Tr_{L/F}\big(a\rho(h)\rho(w)^{-1}\big) \\
    & = (t_\rho(a\rho(h)))_w.
\end{align*}
The identity for the other side is proven analogously.
\end{proof}

We can use this lemma to prove that the twisted trace code is precisely a code, that is, a two-sided ideal of $F^\alpha G$.

\begin{lemma}\label{lem:trace}
$T_\rho$ is a two-sided ideal of $F^\alpha G$ of dimension $m$. Moreover, as a left $F^\alpha G$-module, $T_\rho$ is isomorphic to $L$, on which $\ov g$ acts as multiplication by $\rho(g)$. In particular $T_\rho$ is a minimal left ideal of $F^\alpha G$, and its endomorphism ring is isomorphic to $L$.
\end{lemma}

\begin{proof}
Recall that $t_\rho$ is $F$-linear, so $T_\rho$ is an $F$-subspace of
$F^\alpha G$. By the identity of lemma \ref{lem:tro}, $\ov h\,T_\rho\subseteq T_\rho$ and
$T_\rho\,\ov h\subseteq T_\rho$ for every $h\in G$. As the elements $\ov h$ generate
$F^\alpha G$, it follows that $T_\rho$ is a two-sided ideal of $F^\alpha G$.

Now, we claim that $t_\rho$ is a linear isomorphism from $L$ to $T_\rho$, so they are isomorphic as $F$-vector spaces. As $t_\rho$ is onto $T_\rho$, we just need to prove that $t_\rho$ is injective.  

Let $a\in L$ with $t_\rho(a)=0$, that is,
$\Tr_{L/F}\big(a\rho(g)^{-1}\big)=0$ for all $g\in G$. Note that
$\rho(g)^{-1}=\alpha(g^{-1},g)^{-1}\rho(g^{-1})$, and hence
$\Tr_{L/F}\big(a\rho(g)\big)=\Tr_{L/F}\big(\alpha(g,g^{-1})a\rho(g^{-1})^{-1}\big)=0$ for all $g\in G$. Since $L=F[\rho(G)]$ is the $F$-span
of $\rho(G)$, every $y\in L$ can be written as $y=\sum_{g}b_g\,\rho(g)$ with
$b_g\in F$, and then
\[
\Tr_{L/F}(ay)=\sum_{g\in G} b_g\,\Tr_{L/F}\big(a\rho(g)\big)=0.
\]
However, the trace form is nondegenerate \cite[Section 2.3]{LN}, so $\Tr_{L/F}(ay)=0$ for all
$y\in L$ implies that $a=0$, which proves the claim. Therefore $t_\rho$ is a linear 
isomorphism from $L$ onto $T_\rho$, and $\dim_F T_\rho=[L:F]=m$.

Now consider $L$ as a left $F^\alpha G$-module where $\ov g$ acts as multiplication by $\rho(g)$. Note that it is indeed a module, since $\rho$ is an $\alpha$-character. By the identity above, we have that $\ov h t_\rho(a)=t_\rho\big(a\rho(h)\big)= t_\rho\big(\rho(h)a\big) = t_\rho\big(\ov h \cdot a\big)$, so $t_\rho$ respect the action of $G$ over $L$ and therefore it is an isomorphism of left $F^\alpha G$-modules from $L$ onto $T_\rho$. In addition, the submodules of $L$ are the $F$-subspaces closed under multiplication by every $\rho(g)$, and as $L=F[\rho(G)]$, the only possibility are $0$ and $L$. Therefore $T_\rho$ is a minimal left ideal of $F^\alpha G$. 

Finally, the module endomorphisms of $L$ are the $F$-linear maps commuting with the multiplication by every $\rho(g)$, that is, the $L$-linear maps $L\to L$. Those are precisely the multiplications by elements of $L$ and as a result, the endomorphism ring of $T_\rho$ is isomorphic to $L$. \qed
\end{proof}

\begin{example}
    Observe that if
$\alpha=1$ and $G=\langle g\rangle$ is cyclic of order $n$, then $\rho$ is a group
homomorphism $G\to L^*$, so $\rho(g^i)=\zeta^i$ with $\zeta=\rho(g)$ an $n$-th root of
unity. In this case
\[
T_\rho=\Big\{\big(\Tr_{L/F}(a\,\zeta^{-i})\big)_{0\le i<n}:a\in L\Big\},
\]
which is precisely the usual trace description of an irreducible cyclic code
(see \cite[Chapter 8]{MS}).
\end{example}

The following result shows that, up to monomial equivalence, every twisted trace code is a constacyclic code.

\begin{proposition}\label{prop:cyclic}
Let $n=|G|$. Then there exist a cocycle $\alpha'$ of the cyclic group $C_n$ and an $\alpha'$-character $\rho'\colon C_n\to L^*$ such that $T_\rho$ is monomially equivalent to the twisted trace code $T_{\rho'}$. As a result, $T_\rho$ is monomially equivalent to an irreducible constacyclic code of length $n$ and dimension $m$.
\end{proposition}

\begin{proof}
The map $\bar\rho\colon G\to L^*/F^*$, $g\mapsto\rho(g)F^*$ is a group homomorphism, as $\bar\rho(g) \bar\rho(h) = \alpha(g,h) \bar\rho(gh) = \bar\rho(gh) F^*$. Let us denote $Z=Im(\bar\rho)$. We know that $L^*$ is a cyclic group, so $Z\leq L^* / F^*$ will also be a cyclic group. In addition, $|Z|$ divides $|G| =n$, so there exists an epimorphism $\pi\colon C_n\to Z$. All the fibres of $\bar\rho$ and $\pi$ have $n/|Z|$ elements, so there is a bijection $\varphi\colon G\to C_n$ with $\pi(\varphi(g))=\bar\rho(g)$ for all $g\in G$.
    
 For each $x\in C_n$, choose $\rho'(x)\in L^*$ a representative of the class $\pi(x)\in Z$, with $\rho'(1)=1$ and let $\alpha'(x,y)=\rho'(x)\rho'(y)\rho'(xy)^{-1}$. Note that $\rho'(x)\rho'(y)F^*= \rho'(xy)F^*$, so we have that $\alpha'(x,y)F^* = \rho'(x)\rho'(y)\rho'(xy)^{-1} F^* = F^*$, and therefore it lies in $F^*$, and it is a $2$-cocycle as it is the coboundary of $\rho'$ in $L^*$. Moreover, $F^*\rho(g) =\bar\rho(g) = \pi ( \varphi(g))= F^* \rho'(\varphi(g))$, so $F[\rho'(C_n)]=F[\rho(G)]=L$. 

Finally,  $\rho'(\varphi(g))F^*=\pi(\varphi(g)) =\overline{\rho}(g)=\rho(g)F^*$, so there exists $c_g\in F^*$ such that $\rho(g)=c_g\,\rho'(\varphi(g))$, from where
\[
\Tr\big(a\,\rho(g)^{-1}\big)=c_g^{-1}\,\Tr\big(a\,\rho'(\varphi(g))^{-1}\big)\qquad(a\in L,\ g\in G),
\]
And therefore $T_\rho$ is obtained from $T_{\rho'}$ through the bijection $\varphi$ and the scalars $c_g^{-1}$, so they are monomially equivalent codes. For the last assertion, note that $F^{\alpha'}C_n$ is a commutative algebra because the cocycle $\alpha'$ is symmetric, and $T_{\rho'}$ is a minimal ideal of $F^{\alpha'}C_n$ by Lemma \ref{lem:trace}. Finally, by Example \ref{ex:consta}, there is an isomorphism $F[x]/(x^n-c)\to F^{\alpha'}C_n$ for some $c\in F^*$ which only rescales the basis, and we conclude that $T_\rho$ is monomially equivalent to an irreducible constacyclic code of length $n$ and dimension $m$. \qed
\end{proof}

We can now prove our main result.

\begin{theorem} \label{thm:maindim2and3}
Let $I$ be a nonzero ideal of $F^\alpha G$ with $\dim_FI\le3$.
\begin{enumerate}[(a)]
\item If $\alpha$ is a coboundary, then $I$ is monomially equivalent to an abelian group code.
\item If $\alpha$ is not a coboundary, then $m=\dim_FI\in\{2,3\}$, and there are a field $L\supseteq F$ with $[L:F]=m$ and an $\alpha$-character $\rho\colon G\to L^*$ such that
\[
I=\Big\{\sum_{g\in G}\Tr_{L/F}\big(a\,\rho(g)^{-1}\big)\,\ov g\ :\ a\in L\Big\}.
\]
Moreover, $I$ is monomially equivalent to an irreducible constacyclic code of length $|G|$, that is, to a minimal ideal of $F[x]/(x^{|G|}-c)$ for some $c\in F^*$.
\end{enumerate}
\end{theorem}

\begin{proof}
    (a) If $\alpha$ is a coboundary, Proposition \ref{prop:graded} gives a monomial equivalence between $I$ and an ideal of $FG$ of dimension at most $3$, which is permutation equivalent to an abelian group code by \cite[Theorem 1]{GGMMM}.

    (b) Suppose that  $\alpha$ is not a coboundary. Then by Lemma \ref{lem:onedim} there is no ideal of dimension $1$, so $m=\dim_FI\in\{2,3\}$. In addition, if $I$ was not simple as a left $F^\alpha G$-module, it would have a left submodule of dimension $1$, or $m=3$ and a left submodule $J$ of dimension $2$, and therefore $I/J$ would be a left $F^\alpha G$-module of dimension $1$. As a result, by the  lemma \ref{lem:onedim} it must be simple. By Schur's lemma \cite[Proposition XVII.1.1]{Lang} the endomorphism ring  $D=\End({}_{F^\alpha G}I)$ of $I$ is a division ring, and as it must be finite, it is a field by Wedderburn's theorem \cite[Theorem 3.1.1]{Herstein}, and $D$ contains $F$ as the scalar multiplications. Note that $F^\alpha G$ is also included in $D$ as right multiplication, as $I$ is an ideal. If $D$ were equal to $F$, then multiplication by an element of $F^\alpha G$ would be a scalar multiplication, and therefore every one-dimensional subspace of $I$ would be a right ideal of $F^\alpha G$, which is impossible by Lemma \ref{lem:onedim}. Therefore $[D:F]\ge2$, and since $\dim_FI=[D:F]\cdot\dim_DI\le3$, we get $\dim_DI=1$ and $[D:F]=m$.
    
    We have that $I$ is one-dimensional over the field $D$, so $I=Dx$ for all $0\not = x\in I$. Also, by definition, left multiplication by $\bar g$ commutes with the elements of $D$ for all $g\in G$, so it is a $D$-linear map of the one-dimensional $D$-vector space $I$,so there exists $\rho(g)\in D^*$ such that $\ov g\,x=\rho(g)x$ for all $x\in I$. With the same proof as in Lemma \ref{lem:onedim}, it is easy to check that $\rho$ is an $\alpha$-character. Reasoning as in the previous proof, the subfield $F[\rho(G)]$ of $D$ acts on $I$ and every $F[\rho(G)]$-subspace of $I$ is also a left ideal of $F^\alpha G$. As $I$ is simple, $\dim_{F[\rho(G)]}I=1$ and $F[\rho(G)]=D$.

    It remains to compute the coordinates of the elements of $I$ to prove that ${I=\Big\{\sum_{g\in G}\Tr_{L/F}\big(a\,\rho(g)^{-1}\big)\,\ov g\ :\ a\in L\Big\}}$. To do so, fix $0\ne x_0\in I$, so that $I=Dx_0$, and let $\varepsilon(y)=y_1$ be the coefficient of $\ov1$. The map $D\longrightarrow F$ given by $\delta\mapsto\varepsilon(\delta x_0)$ is $F$-linear, and $D$ is a finite field, so by \cite[Section 2.3]{LN} there exist $\theta\in D$ with $\varepsilon(\delta x_0)=\Tr_{D/F}(\theta\delta)$ for all $\delta\in D$. 

    Now let $y=\delta x_0 \in I$ and $g\in G$. We have $\ov{g^{-1}}\,y=\rho(g^{-1})\delta x_0$ and $\rho(g)\rho(g^{-1})=\alpha(g,g^{-1})$, so by \eqref{eq:coef}
\[
\alpha(g^{-1},g)\,y_g=\varepsilon\big(\ov{g^{-1}}\,y\big)=\varepsilon\big(\rho(g^{-1})\delta x_0\big)=\Tr_{D/F}\big(\theta\delta\rho(g^{-1})\big)=\alpha(g,g^{-1})\Tr_{D/F}\big(\theta\delta\rho(g)^{-1}\big).
\]
Using \eqref{eq:coef} again, we obtain $y_g=\Tr_{D/F}(\theta\delta\rho(g)^{-1})$ for all $g\in G$. In particular $\theta\ne0$ as $I\ne0$. As $\delta$ runs over $D$ so does $\theta\delta$, and therefore $I=T_\rho$ with $L=D$. The last assertion of (b) is Proposition \ref{prop:cyclic}. \qed
\end{proof}

\begin{corollary}\label{cor:noncomm}
Suppose that $g,h\in G$ commute but $\ov g\,\ov h\ne\ov h\,\ov g$ in $F^\alpha G$. Then $F^\alpha G$ has no ideals of dimension $1$, $2$ or $3$.
\end{corollary}

\begin{proof}
    $\ov g\,\ov h\ne\ov h\,\ov g$ means in particular that $\alpha(g,h)\ne\alpha(h,g)$. This implies that $\alpha$ is not a coboundary in any field, as $t(g)t(h)t(gh)^{-1}$ is symmetric for all $t:G\longrightarrow L$ with $L$ a field. Therefore by Lemma \ref{lem:onedim} there is no ideal of dimension $1$. If there was an ideal of dimension $2$ or $3$, by Theorem \ref{thm:maindim2and3}(b) there would be an $\alpha$-character $\rho\colon G\to L^*$, and as in the previous case $\alpha(g,h)=\rho(g)\rho(h)\rho(gh)^{-1}$ is symmetric, which is a contradiction.
\end{proof}

\section{The tetracode}\label{sec:examples}

In this section we will prove that not all twisted group codes of dimension at most $3$ are monomially equivalent to an abelian group code. In fact, we will prove that there exist an infinity family of twisted group code that are not monomially equivalent to any group codes. We will provide such family for both codes in dimension $2$ and in dimension $3$.

For the rest of the section, let $k\in\{2,3\}$, and for $j\ge0$ let $N_j=(q^j-1)/(q-1)$, so that $N_0=0$, $N_1=1$, $N_2=q+1$ and $N_3=q^2+q+1$. Note that $N_j-N_{j-1}=q^{j-1}$ for $j\ge1$, and that $N_j$ is the number of element over the projective space $\mathbb P ^{j-1}(F)$.
 
To compare codes we will look at their coordinates as linear maps, that is, let $\mathcal C\subseteq F^X$ be a linear code. For $x\in X$ we denote by $\varepsilon_x\colon\mathcal C\to F$, $c\mapsto c_x$, the coordinate map at $x$, and by $\mathcal C^*$ it dual space, that is, the space of linear maps $\mathcal C\to F$. Note that have only define $\varepsilon_x$ on $\mathcal{C}$, so it could happen that $\varepsilon_x=0$ while in the general space $F^X$ that is not true in general.
 
The following lemma shows that all $\left[ r N_k, \, k, \, r q^{k-1} \right]_q$ codes are monomially equivalent.
 
\begin{lemma}\label{lem:simplex}
Let $r\ge1$ and let $\mathcal C\subseteq F^X$ be a linear code of length $n=rN_k$ and dimension $k$.
\begin{enumerate}[(a)]
\item The minimum distance of $\mathcal C$ is at most $rq^{k-1}$.
\item If the minimum distance of $\mathcal C$ is $rq^{k-1}$, then every nonzero codeword has weight $rq^{k-1}$, and for every nonzero $\varphi\in\mathcal C^*$ there are exactly $r$ elements $x\in X$ with $\varepsilon_x\in F^*\varphi$.
\item All the codes of length $rN_k$, dimension $k$ and minimum distance $rq^{k-1}$ are monomially equivalent.
\end{enumerate}
\end{lemma}
 
\begin{proof}
(a) Note that this is precisely the Plotkin bound for linear codes \cite{HuffmanPless2003}. Indeed, for a linear code of length $n$ and dimension $k$ over $F$ we have that its minimum distance is at most $nq^{k-1}(q-1)/(q^k-1)$, and for $n=rN_k$ this is exactly $rq^{k-1}$, as $N_k(q-1)=q^k-1$.

(b) Suppose that the minimum distance of $\mathcal C$ is $rq^{k-1}$, so $\mathcal C$ satisfy the Plotkin bound. Note that if $\varepsilon_x\ne0$, then its kernel has dimension $\dim C - \dim F= k-1$, so there are exactly $q^k-q^{k-1}$ codewords with $c_x\ne0$. Therefore the sum of the weights of the nonzero codewords is $n'(q^k-q^{k-1})$, where $n'$ is the number of $x\in X$ with $\varepsilon_x\ne0$, and this is at most $(q^k-1)rq^{k-1}$. As all the $q^k-1$ nonzero codewords have weight at least $rq^{k-1}$, we have that $n'=n$ and all of them have weight exactly $rq^{k-1}$. In particular, $\mathcal C$ has no zero coordinates and only one nonzero weight.

Therefore, by Bonisoli's theorem \cite{Bonisoli}, $\mathcal C$ is monomially equivalent to the $t$-fold replicated simplex code of dimension $k$ for some $t\ge1$, that is, the code with a generator matrix whose columns are $t$ copies of a nonzero vector of each one-dimensional subspace of $F^k$. This code has length $tN_k$, so comparing lengths we get $t=r$. Now, note that in the $r$-fold replicated simplex code we have that for every nonzero $\varphi\in\mathcal C^*$ there are exactly $r$ coordinates $x$ with $\varepsilon_x\in F^*\varphi$. This property is preserved by monomial equivalence, and therefore it also holds for $\mathcal C$, which proves (b).

(c) It follows from the proof of (b), as every code of length $rN_k$, dimension $k$ and minimum distance $rq^{k-1}$ is monomially equivalent to the $r$-fold replicated simplex code of dimension $k$. \qed
\end{proof}

Note that the codes of Lemma \ref{lem:simplex}(c) are the replicated simplex codes, and for $r=1$ they are the simplex codes \cite{HuffmanPless2003}. For our purpose, note that for $k=2$, $q=3$ and $r=1$, Lemma \ref{lem:simplex}(c) says that every $[4,2,3]_3$ code is monomially equivalent to the tetracode.

The following lemma will be useful to compare twisted group codes with group codes. In particular, it describes the class of equivalence of the equivalence relation of the previous proof.
 
\begin{lemma}\label{lem:orbit}
Let $I$ be a nonzero ideal of $FG$, and for $g\in G$ let $\rho_I(g)\colon I\to I$ be the left multiplication by $g$. Then $\varepsilon_g=\varepsilon_1\circ\rho_I(g^{-1})$ for every $g\in G$. Moreover, the set
\[
S=\{s\in G:\ \varepsilon_1\circ\rho_I(s)\in F^*\varepsilon_1\}
\]
is a subgroup of $G$, and for all $g,h\in G$ we have $\varepsilon_h\in F^*\varepsilon_g$ if and only if $h\in gS$.
\end{lemma}
 
\begin{proof}
Let $g\in G$ and $y \in I$. By \eqref{eq:coef} with $\alpha=1$ we have that the coefficient of $1$ in $\overline{g^{-1}}y$ is $y_g$, so $\varepsilon_g(y)=\varepsilon_1(g^{-1}y)= \varepsilon_1\circ\rho_I(g^{-1})$ as we wanted to prove. In particular if $\varepsilon_1=0$ then $\varepsilon_g(y)=\varepsilon_1\circ\rho_I(g^{-1})=0$ for all $g\in G, y\in I$ so $y=0$ and $I=0$. Therefore there exist at least one element in $I$ with nonzero coordinate in $1$.

Now we will prove that $S$ is a group. First, note that $\varepsilon_1\circ\rho_I(1) = \varepsilon_1\in F^*\varepsilon_1$ so $1\in S$. If $s,t\in S$, say $\varepsilon_1\circ\rho_I(s)=a\,\varepsilon_1$ and $\varepsilon_1\circ\rho_I(t)=b\,\varepsilon_1$, then $\varepsilon_1\circ\rho_I(t^{-1})=b^{-1}\,\varepsilon_1$ and
$\varepsilon_1\circ\rho_I(st^{-1})= \big(\varepsilon_1\circ\rho_I(s)\big)\circ\rho_I(t)=a\,\varepsilon_1\circ\rho_I(t^{-1})=ab^{-1}\,\varepsilon_1,$, so $st^{-1}\in S$ and $S$ is a subgroup.
 
Finally, let $g,h\in G$ such that $\varepsilon_h\in F^*\varepsilon_g$. Then, there exists $\lambda\in F^*$ such that $\varepsilon_h\in \lambda\varepsilon_g$. By the first formula, that means that $\varepsilon_1\circ\rho_I(h^{-1})=\lambda\,\varepsilon_1\circ\rho_I(g^{-1})$, and composing with $\rho_I(g)$ on the right we get that $\varepsilon_1\circ\rho_I(h^{-1}g)=\lambda\,\varepsilon_1 \in F^*\varepsilon_1$ so $h^{-1}g\in S$. As all implications are reversible, $\varepsilon_h\in F^*\varepsilon_g$ if and only if $h^{-1}g\in S$, that is, if and only if $h\in gS$. \qed
\end{proof}

We need the following technical result for finite fields.
\begin{lemma}\label{lemm:eta}
    Let $L=\F_{q^k}$, let $r\ge1$ and $n=rN_k$. let $\eta,a\in L^*$ such that $\eta F^*$ generates $L^*/F^*$, and let $H_a=\{y\in L:\Tr_{L/F}(ay)=0\}$. Then exactly $rN_{k-1}$ of the elements $\Tr_{L/F}(a\eta^{-i})$, $0\le i<n$, are zero.
\end{lemma}
\begin{proof}
Let $L=\F_{q^k}$, that is, $[L:F]=k$ and $L^*/F^*$ has $(q^k-1)/(q-1)=N_k$ elements. Note that scalar multiplication defines an action of the multiplicative group $F^*$ on the $F$-vector space $L$. Therefore, if $H$ is an $F$-subspace, it is invariant under this action, and  for any $x \in H \setminus \{0\}$ we have that $x F^* \subseteq H \setminus \{0\}$. Because group orbits partition the set, $H \setminus \{0\}$ is a disjoint union of these cosets.

For $a\in L^*$ let $H_a=\{y\in L:\Tr_{L/F}(ay)=0\}$. The trace $\Tr_{L/F}\colon L\to F$ is $F$-linear and nonzero \cite[Section 2.3]{LN}, so $y\mapsto\Tr_{L/F}(ay)$ is a nonzero $F$-linear map. Therefore $H_a$ is an $F$-subspace of $L$ of dimension $k-1$. As it is a $F$-subspace, $H_a\setminus\{0\}$ is a union of classes of $L^*/F^*$. Each class has $q-1$ elements and $|H_a\setminus\{0\}|=q^{k-1}-1=N_{k-1}(q-1)$, so $H_a\setminus\{0\}$ is the union of exactly $N_{k-1}$ classes of $L^*/F^*$.

Now we can prove the result. Note that $\eta F^*$ generates $L^*/F^*$ is equivalent to
say that for every integer $j$, $\eta^j\in F^*$ if and only if $N_k$ divides $j$. Now, for $0\le i,i'<n$, the elements $\eta^{-i}$ and $\eta^{-i'}$ are in the same class of $L^*/F^*$ if and only if $\eta^{i'-i}\in F^*$, that is, if and only if $i\equiv i'\pmod{N_k}$. Therefore the classes of $\eta^{-i}$ for $0\le i<N_k$ are pairwise different, and as there are $N_k$ classes in $L^*/F^*$, they are all the classes. Since $n=rN_k$, each class contains $\eta^{-i}$ for exactly $r$ values of $i\in\{0,\dots,n-1\}$. Now, $\Tr_{L/F}(a\eta^{-i})=0$ if and only if $\eta^{-i}\in H_a$, that is, if and only if the class of $\eta^{-i}$ is one of the $N_{k-1}$ classes contained in $H_a$. This happens for exactly $rN_{k-1}$ values of $i$, which proves the claim.
\end{proof}
 
We can now prove the main result of this section.
 
\begin{theorem}\label{thm:simplex}
Let $k\in\{2,3\}$, let $r\ge1$ and let $n=rN_k$.
\begin{enumerate}[(i)]
\item There is a minimal ideal of $F[x]/(x^n-c)$, for some $c\in F^*$, which is a code of length $n$, dimension $k$ and minimum distance $rq^{k-1}$.
\item If $k$ does not divide $q-1$, there is also a minimal ideal of $FC_n$ of length $n$, dimension $k$ and minimum distance $rq^{k-1}$.
\item If $k$ divides $q-1$ and $k$ does not divide $r$, then every ideal of a group algebra $FG$ with $|G|=n$ and dimension $k$ has minimum distance at most $rq^{k-1}-1$. Consequently, the code of (i) is not monomially equivalent to any group code.
\end{enumerate}
\end{theorem}

\begin{proof}

(i) Let $\zeta$ be a generator of the cyclic group $L^*$. Then clearly $\zeta F^*$ generates $L^*/ F^*$ and we can apply the previous lemma. In particular $\zeta^{N_k}\in F^*$ and $c=\zeta^n=(\zeta^{N_k})^r$ lies in $F^*$. Let $C_n=\langle g\rangle$ and let $\alpha$ be the cocycle of Example \ref{ex:consta} for this $c$.  It is a direct computation to check that $\rho(g^i)=\zeta^i$ for $0\le i<n$ is an $\alpha$-character. In addition, $F[\rho(C_n)] =F[\zeta] = L$. Therefore, by Lemma \ref{lem:trace}, $T_\rho$ is a two-sided ideal of $F^\alpha C_n$ of dimension $k$ which is a minimal left ideal, and so it is also a minimal ideal. By Example \ref{ex:consta}, $F^\alpha C_n\cong F[x]/(x^n-c)$ by an isomorphism which only rescales the basis, and therefore $T_\rho$ is a minimal ideal of $F[x]/(x^n-c)$, up to monomial equivalence.

Now we compute the weights for each $a\in L^*$. To do so, we have that the coordinate of $t_\rho(a)$ at $g^i$ is $\Tr_{L/F}(a\rho(g^i)^{-1})=\Tr_{L/F}(a\zeta^{-i})$. Therefore, by the previous lemma $t_\rho(a)$ has exactly $rN_{k-1}$ zero coordinates and its weight is
\[
n-rN_{k-1}=r(N_k-N_{k-1})=rq^{k-1}.
\]
Therefore the minimum distance of $T_\rho$ is $rq^{k-1}$.
 
(ii) Now assume that $k$ does not divide $q-1$. Then we have that $N_k=1+q+\dots+q^{k-1}\equiv k\pmod{q-1}$, and therefore $\gcd(N_k,q-1)=\gcd(k,q-1)=1$. Now let $\xi\in L^*$ be an element of order $N_k$ and we will prove that $F[\xi]=L$. To do so, note that $\xi^j\in F^*$ if and only if $\xi^{j(q-1)}=1$, that is, if and only if $N_k$ divides $j(q-1)$, and as $\gcd(N_k,q-1)=1$, that is true if and only if $N_k$ divides $j$. Therefore $\xi F^*$ generates $L^*/F^*$, and in particular $\xi\notin F$, as $N_k>1$. Then $F[\xi]$ is a subfield of $L$ different from $F$, for which $[F[\xi]:F]$ divides $[L:F]=k$, which is prime, so we get that $F[\xi]=L$. 

Now let $\chi\colon C_n\to L^*$, $g^i\mapsto\xi^i$. The map is well defined as $\xi^n=(\xi^{N_k})^r=1$, and also it is easy to check that it is a group homomorphism with $F[\chi(C_n)]=F[\xi]=L$. By Lemma \ref{lem:trace}, $T_\chi$ is a minimal ideal of $FC_n$ of dimension $k$. Finally, for $a\in L^*$ the coordinate of $t_\chi(a)$ at $g^i$ is $\Tr_{L/F}(a\xi^{-i})$, so by Lemma \ref{lemm:eta} with $\eta=\xi$, every nonzero codeword of $T_\chi$ has weight $n-rN_{k-1}=rq^{k-1}$.
 
(iii) Suppose that $I$ is an ideal of $FG$ with $|G|=n$, of dimension $k$ and minimum distance at least $rq^{k-1}$. By Lemma \ref{lem:simplex}(a), we have that the minimum distance of $I$ is exactly $rq^{k-1}$. 

Now, let $S$ be the subgroup of Lemma \ref{lem:orbit}. Consider the equivalence relation on $G$ given by $g\sim h$ if $\varepsilon_h\in F^*\varepsilon_g$. By Lemma \ref{lem:orbit}, the class of $g$ is the left coset $gS$, which has $|S|$ elements. On the other hand, by Lemma \ref{lem:simplex}(b) applied to $\varphi=\varepsilon_g$, the class of $g$ has exactly $r$ elements. Therefore $|S|=r$.

Now, $q\equiv1\pmod k$, so $N_k=1+q+\dots+q^{k-1}\equiv k\equiv0\pmod k$, and $k$ divides $n=rN_k=|G|$. As $k$ is prime, by Cauchy's theorem there is an element $x\in G$ of order $k$. Consider the linear map $T\colon I^*\to I^*$, $\psi\mapsto\psi\circ\rho_I(x)$. We have that $T^k(\psi)=\psi\circ\rho_I(x)^k=\psi\circ\rho_I(x^k)=\psi\circ\rho_I(1)=\psi$ for all $\psi\in I^*$, so $T^k=\mathrm{id}$. As $k$ divides $q-1$, the polynomial $X^k-1$ has $k$ different roots in $F$, so $T$ is diagonalizable over $F$, and, in particular, there are $0\ne\varphi\in I^*$ and $\lambda\in F^*$ with $\varphi\circ\rho_I(x)=\lambda\,\varphi$.
 
 By Lemma \ref{lem:simplex}(b), for such $,\varphi$ there exits $g\in G$ with $\varepsilon_g\in F^*\varphi$, say $\varepsilon_g=\nu\varphi$ with $\nu\in F^*$. Multiplying the last identity by $\nu$ we get $\varepsilon_g\circ\rho_I(x)= \nu\varphi\circ\rho_I(x)=\lambda\,\nu\varphi = \lambda\,\varepsilon_g$, and by Lemma \ref{lem:orbit} this means that
\[
\varepsilon_1\circ\rho_I(g^{-1})\circ\rho_I(x)=\lambda\,\varepsilon_1\circ\rho_I(g^{-1}).
\]
Composing with $\rho_I(g)$ on the right, we get $\varepsilon_1\circ\rho_I(g^{-1}xg)=\lambda\,\varepsilon_1$, that is, $g^{-1}xg\in S$. But $ord(g^{-1}xg) = ord(x)=k$, so by Lagrange's theorem $k$ divides $|S|=r$, which is a contradiction. Therefore every ideal of $FG$ of dimension $k$ has minimum distance at most $rq^{k-1}-1$.

Finally, let $\mathcal C$ be the code of (i), and suppose that it is monomially equivalent to an ideal $J$ of a group algebra $FH$. Monomially equivalent codes have the same length, dimension and minimum distance, so $|H|=n$, $\dim_FJ=k$ and the minimum distance of $J$ is $rq^{k-1}$, which contradicts what we have just proved.\qed
\end{proof}

\begin{remark}
    Note that $k=2$ divides $q-1$ for all $q$ odd. Therefore, Theorem \ref{thm:simplex} ii) applies $k=2$ and $q$ even or $k=3$ and $q\not =1 \mod 3$.
\end{remark}
 
Therefore, in case (iii) the codes of Theorem \ref{thm:simplex} appear in case (b) of Theorem \ref{thm:main}, but they are not monomially equivalent to any group code. We write explicitly the parameters in each dimension.
 
\begin{corollary}\label{cor:families}
\begin{enumerate}[(a)]
\item Let $q$ and $r$ be odd. There is a $[(q+1)r,2,qr]_q$ irreducible constacyclic code which is not monomially equivalent to any group code, and every two-dimensional group code of length $(q+1)r$ over $\F_q$ has minimum distance at most $qr-1$.
\item Let $q\equiv1\pmod 3$ and let $r$ be an integer not divisible by $3$. There is a $[(q^2+q+1)r,3,q^2r]_q$ irreducible constacyclic code which is not monomially equivalent to any group code, and every three-dimensional group code of length $(q^2+q+1)r$ over $\F_q$ has minimum distance at most $q^2r-1$.
\end{enumerate}
By Lemma \ref{lem:simplex}(a), in both cases the constacyclic code has the largest possible minimum distance for its length and dimension.
\end{corollary}
 
\begin{proof}
It is Theorem \ref{thm:simplex} with $k=2$ and $k=3$, respectively. \qed
\end{proof}
 
For example, if $q=5$ we get a $[6,2,5]_5$ ideal of $\F_5[x]/(x^6-2)$, whereas every two-dimensional ideal of $\F_5C_6$ and of $\F_5S_3$ has minimum distance at most $4$. In dimension $3$, the algebra $\F_4[x]/(x^{21}-\omega)$, with $\omega$ a generator of $\F_4^*$, contains ideals which are $[21,3,16]_4$ codes, and $\F_7[x]/(x^{57}-c)$ contains ideals which are $[57,3,49]_7$ codes for a suitable $c\in\F_7^*$. On the other hand, the largest minimum distance of a three-dimensional cyclic code is $12$ for length $21$ over $\F_4$, and $45$ for length $57$ over $\F_7$.
 
\begin{remark}\label{rem:sharp}
The bound of Corollary \ref{cor:families}(a) is sharp for $r=1$. To see this, let $q$ be odd, let $L=\F_{q^2}$ and let $\xi\in\F_{q^2}^*$ be an element of order $q+1$, and let $\chi\colon C_{q+1}\to\F_{q^2}^*$ be given by $g^i\mapsto\xi^i$. Note that $\xi\notin\F_q$, as $\xi^{q}=\xi^{-1}\neq\xi$ , so $\F_{q^2}=\F_{q}[\xi]$. In addition, $\chi$ is a group homomorphism, so it is a character for the trivial cocycle, and by Lemma \ref{lem:trace} we have that $T_\chi$ is a two-dimensional ideal of $\F_qC_{q+1}$. Also, note that $\xi^j\in\F_q^*$ if and only if $\xi^{jq}=\xi^j$, that is, $(q+1)| j(q-1)$, which is equivalent to $j$ been a multiple of $(q+1) / (gcd(q+1,q-1)= (q+1)/2$. 

Now let $a\in L^*$. As in the proof of Theorem \ref{thm:simplex}, $H_a\setminus\{0\}$ is exactly one class of $L^*/\F_q^*$, and the coordinate of $t_\chi(a)$ at $g^i$ vanishes if and only if $\xi^{-i}\in H_a$. By the previous paragraph, $\xi^{-i}$ and $\xi^{-i'}$, with $0\le i,i'\le q$, are in the same class if and only if $i\equiv i'\pmod{(q+1)/2}$, so every class contains at most two of the elements $\xi^{-i}$. Therefore $t_\chi(a)$ has weight at least $q-1$. Finally, if $a\neq0$ satisfies $\Tr(a)=0$, then $1$ and $\xi^{-(q+1)/2}=-1$ are in $H_a$. Hence $t_\chi(a)$ has exactly two zero coordinates, at $g^0$ and $g^{(q+1)/2}$, and its weight is $q-1$. As a result, $T_\chi$ is a two-dimensional group code of length $q+1$ and minimum distance $q-1$.
\end{remark}

\begin{remark}
Note that the only place where we have used that $k\in\{2,3\}$ is that $k$ is prime. In general, the same proof shows that for every $k\ge2$ and $r\ge1$ there is an irreducible constacyclic code of length $rN_k$, dimension $k$ and minimum distance $rq^{k-1}$. In addition, every group code with the same length and dimension has smaller minimum distance in case there is a prime $\ell$ which divides $\gcd(k,q-1)$ but not $r$.
\end{remark}

\begin{corollary}\label{tetracode}
The ternary tetracode, of parameters $[4,2,3]_3$, is an ideal of the commutative twisted group algebra $\F_3[x]/(x^4+1)$ of the cyclic group of order $4$, but it is not monomially equivalent to any group code.
\end{corollary}
\begin{proof}
We will use the previous theorem with $q=3$, $k=2$ and $r=1$. To do so, note that a generator $\zeta$ of $\F_9^*$ has order $8$, so $\zeta^4=-1$ and the algebra of Theorem \ref{thm:simplex}(i) is $\F_3[x]/(x^4+1)$. By Lemma \ref{lem:simplex}(c), the tetracode is monomially equivalent to the code of Theorem \ref{thm:simplex}(i). In addition, by Theorem \ref{thm:simplex}(iii) every two-dimensional group code of length $4$ over $\F_3$ has minimum distance at most $2$. As a result, the tetracode is not monomially equivalent to any group code. \qed
\end{proof} 

\begin{remark}
We can also describe the tetracode explicitly as an ideal. For this, let $G=\langle g\rangle\cong C_4$ and let $\alpha$ be as in Example \ref{ex:consta} with $c=-1$, so that $\F_3^\alpha G\cong\F_3[x]/(x^4+1)$. Over $\F_3$ we have that $x^4+1=(x^2+x+2)(x^2+2x+2)$, and also $x^2(x^2+x+2)\equiv(x+1)(x^2+x+2)$ modulo $x^4+1$. Therefore the ideal generated by $x^2+x+2$ is $\mathcal C=\{(a+bx)(x^2+x+2):a,b\in\F_3\}$, that is, in coordinates,
\[
\mathcal C=\{(2a,\ a+2b,\ a+b,\ b)\ :\ a,b\in\F_3\},
\]
which is the tetracode \cite{MS}. In the notation of Theorem \ref{thm:main}, we have that $\mathcal C=T_\rho$, where $\rho(g^i)=\zeta^i$ and $\zeta\in\F_9$ is a root of $x^2+2x+2$, which is a primitive $8$th root of unity. Note that the cocycle is not a coboundary. Indeed, $x^4+1$ has no root in $\F_3$, so $\F_3^\alpha G$ has no one-dimensional module by Lemma \ref{lem:onedim}. As a result, the theorem of \cite{GGMMM} fails for twisted group codes already in dimension $2$, and even for a commutative twisted group algebra of a cyclic group.
\end{remark}

However, in case (b) of Theorem \ref{thm:main} the ideal may still coincide with a group code, as it is shown in the following example.

\begin{example}
Let $F=\F_3$, $G=S_3$, and let $\alpha(g,h)=-1$ in case $g$ and $h$ are both odd permutations and $\alpha(g,h)=1$ otherwise. Note that $\alpha$ is the inflation along the sign homomorphism $S_3\to C_2$ of the cocycle of $\F_3[x]/(x^2+1)$, so it is a cocycle.

Now we will see that it is not a coboundary. To do so, note that if $\tau$ be a transposition, then the restriction of $\alpha$ to $\langle\tau\rangle$ gives $\ov\tau^{\,2}=-1$, so $\F_3^\alpha\langle\tau\rangle\cong\F_3[x]/(x^2+1)$. Note that this algebra has no one-dimensional module, because $x^2+1$ has no root in $\F_3$. Therefore, by Lemma \ref{lem:onedim}, the restriction of $\alpha$ to $\langle\tau\rangle$ is not a coboundary, and then neither is $\alpha$.

Now let $i\in\F_9$ with $i^2=-1$. Then the map $\rho$ with $\rho(g)=1$ for $g$ even and $\rho(g)=i$ for $g$ odd is an $\alpha$-character. In addition, $\Tr_{\F_9/\F_3}(u+vi)=2u$ for all $u,v\in\F_3$. It follows that $T_\rho$ consists of the vectors which are constant on $A_3$ and on $S_3\setminus A_3$. Finally, this $[6,2,3]_3$ code coincides with the group code $\F_3S_3\cdot\sum_{h\in A_3}h$, and by Proposition \ref{prop:cyclic} it is also an irreducible constacyclic code of length $6$.
\end{example}

Finally, we show that the assumption that $I$ is a two-sided ideal cannot be dropped in the description of $I$ as a twisted trace code given in Theorem \ref{thm:main}(b). Note that this assumption was used in the proof of Theorem \ref{thm:main} to show that $D\neq F$.

\begin{example}\label{ex:quaternion}
Let $F=\F_3$ and $G=C_2\times C_2=\{1,a,b,ab\}$, and consider the matrices
\[
\mathbf i=\begin{pmatrix}0&1\\-1&0\end{pmatrix},\qquad\mathbf j=\begin{pmatrix}1&1\\1&-1\end{pmatrix},\qquad\mathbf k=\mathbf i\mathbf j=\begin{pmatrix}1&-1\\-1&-1\end{pmatrix}
\]
in $M_2(\F_3)$. A direct computation gives $\mathbf i^2=\mathbf j^2=-1$ and $\mathbf j\mathbf i=-\mathbf k$, and then $\mathbf k^2=-1$ as well. In addition, together with the identity, these matrices form a basis of $M_2(\F_3)$. We also have that the product of any two of the matrices $1,\mathbf i,\mathbf j,\mathbf k$ is, up to the sign, one of them. As a result, $M_2(\F_3)=\F_3^\alpha G$, where $\alpha$ is the cocycle determined by $\ov a=\mathbf i$, $\ov b=\mathbf j$ and $\ov{ab}=\mathbf k$.

Now, the left ideal generated by $1+\mathbf i+\mathbf j$ is
\[
\{(\lambda-\mu,\ \lambda+\mu,\ \lambda,\ \mu)\ :\ \lambda,\mu\in\F_3\}
\]
with respect to the basis $1,\mathbf i,\mathbf j,\mathbf k$. It is again a $[4,2,3]_3$ code, so by Lemma \ref{lem:simplex}(c) it is monomially equivalent to the tetracode. Since every group of order $4$ is abelian, Corollary \ref{tetracode} shows that it is not monomially equivalent to any group code, left or two-sided. On the other hand, we have that $\ov a\,\ov b=-\ov b\,\ov a$, so by Corollary \ref{cor:noncomm} this algebra has no two-sided ideals of dimension less or equal to $3$.
\end{example}

\section{Concluding remarks}\label{sec:concl}

Theorem~\ref{thm:main} reduces the study of twisted group codes of dimension at most $3$ to two families of codes. When the cocycle is a coboundary, the code is monomially equivalent to an abelian group code. In any other case, the code is monomially equivalent to an irreducible constacyclic code.

The second family cannot be avoided as it can be seen in section \ref{sec:examples}. In particular, we have proved that the tetracode $[4,2,3]_3$ cannot be monomially equivalent to any group code. In addition, we prove that every two-dimensional group code of length $4$ over $\mathbb F_3$ has minimum distance at most $2$, so we also show that twisted group codes can have better parameters than all group codes with the same length and dimension.

Several questions remain open. Our results concern two-sided ideals and we do not know whether every one-sided twisted group code of dimension at most $3$ are also monomially equivalent to a one-sided ideal of a twisted group algebra of an abelian group. The example ~\ref{ex:quaternion} points in this direction as the left ideal considered there is a $[4,2,3]_3$ code, hence monomially equivalent to the tetracode \cite{MS}. This is an ideal of the commutative algebra $\mathbb F_3[x]/(x^4+1)$. However, the proof of Theorem~\ref{thm:main} does not apply to one-sided ideals, since it uses that the endomorphism ring $D$ of a two-sided ideal is strictly larger than $F$.

Also it would be of interest to find an analogue of Theorem~\ref{thm:main} for the twisted skew group rings of \cite{BBdlC}, which generalize twisted group algebras.

Finally, it would be of interest to determine for which lengths and dimensions twisted trace codes have better parameters than all group codes, as occurs with the tetracode.

\section*{Declaration of competing interest}
The author declares that he has no known competing financial interests or personal relationships that could have appeared to influence the work reported in this paper.

\section*{Data availability}
No data was used for the research described in the article.

\section*{Acknowledgments}
All the examples in this paper have been verified by computer with assistance from Claude (Anthropic). This research was supported by the Spanish Ministry of Science, Innovation, and Universities under the FPU 2023 grant programme, and the project PID2023-149203NB-I00

\section*{Declaration of Generative AI and AI-assisted technologies in the manuscript preparation process}
During the preparation of this work the author used Claude (Anthropic) to enhance the writing of the manuscript, check the proofs and compute the examples. After using this tool, the author reviewed, verified and edited all mathematical content and takes full responsibility for the content of the published article.

\end{document}